\documentclass{amsart}

\usepackage[mathscr]{eucal}

\newtheorem{theorem}{Theorem}[section]
\newtheorem{lemma}[theorem]{Lemma}
\newtheorem{corollary}[theorem]{Corollary}

\theoremstyle{definition}
\newtheorem{definition}[theorem]{Definition}

\theoremstyle{remark}
\newtheorem{remark}[theorem]{Remark}

\numberwithin{equation}{section}

\newcommand{\xfi}{\mathscr{X}_{\mathscr{C}}}

\newcommand{\mc}{\mathscr{C}}
\newcommand{\md}{\mathscr{D}}
\newcommand{\mw}{\mathscr{W}}
\newcommand{\mv}{\mathscr{V}}

\newcommand{\qn}{q_1 q_2 \cdots q_n}
\newcommand{\qk}{q_1 q_2 \cdots q_k}
\newcommand{\formalsum}{\sum_{n=1}^{\infty} \frac {E_n} {q_1 q_2 \ldots q_n}}

\newcommand{\ab}{(\al,\be)}

\newcommand{\la}{\lambda}
\newcommand{\al}{\alpha}
\newcommand{\ga}{\gamma}

\newcommand{\be}{\beta}

\newcommand{\labeq}[1]{\label{eqn:#1}}
\newcommand{\refeq}[1]{(\ref{eqn:#1})}
\newcommand{\labt}[1]{\label{thm:#1}}
\newcommand{\reft}[1]{Theorem~\ref{thm:#1}}
\newcommand{\labl}[1]{\label{lemma:#1}}
\newcommand{\refl}[1]{Lemma~\ref{lemma:#1}}

\newcommand{\labd}[1]{\label{definition:#1}}

\newcommand{\labc}[1]{\label{coro:#1}}
\newcommand{\refc}[1]{Corollary~\ref{coro:#1}}

\newcommand{\floor}[1]{\lfloor #1 \rfloor} 
\newcommand{\ceil}[1]{\lceil #1 \rceil}

\begin{document}

\title{On winning sets and non-normal numbers}

\author{Bill Mance}
\address{Department of Mathematics, The Ohio State University, Columbus, Ohio 43210-1174}
\email{mance@math.ohio-state.edu}
\thanks{I would like to thank  Vitaly Bergelson, Dimitry Kleinbock, and Jim Tseng for helpful discussions.}

\subjclass[2000]{Primary 11K16, 11A63}

\date{October 12, 2010 and, in revised form, ---- --, ----.}


\keywords{ Cantor series, Normal numbers,  Schmidt games, Winning sets}

\begin{abstract}
In \cite{SchmidtGames}, W. Schmidt proved that the set of non-normal numbers in base $b$ is a {\it winning set}.   We generalize this result by proving that many sets of non-normal numbers with respect to the Cantor series expansion are winning sets.  As an immediate consequence, these sets will be shown to have full Hausdorff dimension.
\end{abstract}

\maketitle


\bibliographystyle{amsplain}

\section{Introduction}
\label{sec:1}
\subsection{Winning sets}
In \cite{SchmidtGames}, W. Schmidt proposed the following game between two players: Alice and Bob.  Let $\al \in (0,1)$, $\be \in (0,1)$, $S \subset \mathbb{R}^n$, and let $\rho(I)$ denote the radius of a set $I$.  Bob first picks any closed interval $B_1\subset \mathbb{R}^n$.  Then Alice picks a closed interval $A_1 \subset B_1$ such that $\rho(A_1)=\al \rho( B_1)$.  Bob then picks a closed interval $B_2 \subset W_1$ with $\rho(B_2)=\be \rho(A_1)$.  After this, Alice picks a closed interval $A_2 \subset B_2$ such that $\rho(A_2)=\al \rho(B_2)$, and so on.  We say that the set $S$ is {\it $(\al,\be)$-winning} if Alice can play so that
\begin{equation}
\bigcap_{n=1}^{\infty} B_n \subset S.
\end{equation}
The set $S$ is {\it $(\al,\be)$-losing} if it is not $(\al,\be)$-winning and {\it $\al$-winning} if it is $\ab$-winning for all $0< \be < 1$.
Winning sets satisfy the following properties:\footnote{See \cite{Dani} and \cite{SchmidtGames}.}
\begin{enumerate}
\item If $S \subset \mathbb{R}^n$ is an $\al$-winning set, then the Hausdorff dimension of $S$ is $n$.
\item The intersection of countably many $\al$-winning sets is $\al$-winning.
\item Bi-Lipshitz homeomorphisms of $\mathbb{R}^n$ preserve winning sets.
\end{enumerate}



\subsection{Normal numbers}

\begin{definition}\labd{1.1} Let $b$ and $k$ be positive integers.  A {\it block of length $k$ in base $b$} is an ordered $k$-tuple of integers in $\{0,1,\ldots,b-1\}$.  A {\it block of length $k$} is a block of length $k$ in some base $b$.  A {\it block} is a block of length $k$ in base $b$ for some integers $k$ and $b$.
\end{definition}

\begin{definition}\labd{1.2} Given an integer $b \geq 2$, the {\it $b$-ary expansion} of a real $x$ in $[0,1)$ is the (unique) expansion of the form
$$
x=\sum_{n=1}^{\infty} \frac {E_n} {b^n}=0.E_1 E_2 E_3 \ldots
$$
such that $E_n$ is in $\{0,1,\ldots,b-1\}$ for all $n$ with $E_n \neq b-1$ infinitely often.
\end{definition}

Denote by $N_n^b(B,x)$ the number of times a block $B$ occurs with its starting position no greater than $n$ in the $b$-ary expansion of $x$.

\begin{definition}\labd{1.3} A real number $x$ in $[0,1)$ is {\it normal in base $b$} if for all $k$ and blocks $B$ in base $b$ of length $k$, one has
$$
\lim_{n \rightarrow \infty} \frac {N_n^{b}(B,x)} {n}=b^{-k}.
$$
\end{definition}

 Borel introduced normal numbers in 1909 and proved that almost all (in the sense of Lebesgue measure) real numbers in $[0,1)$ are normal in all bases.  The best known example of a number that is normal in base $10$ is due to Champernowne \cite{Champernowne}.  The number
$$
H_{10}=0.1 \ 2 \ 3 \ 4 \ 5 \ 6 \ 7 \ 8 \ 9 \ 10  \ 11 \ 12 \ldots ,
$$
formed by concatenating the digits of every natural number written in increasing order in base $10$, is normal in base $10$.  Any $H_b$, formed similarly to $H_{10}$ but in base $b$, is known to be normal in base $b$. Since then, many examples have been given of numbers that are normal in at least one base.  One  can find a more thorough literature review in \cite{DT} and \cite{KuN}.

Suppose that $X=\{x_n\}_{n=1}^{\infty}$ is a sequence of real numbers.  For a positive integer $N$ and some $I \subset [0,1)$, we define $A_N(I,X)$ to be the number of terms $x_n$ with $1 \leq n \leq N$, for which $x_n-\floor{x_n} \in I$.  Thus, we may write
$$
A_N(I,X)=\#\{n  \in [1,N] :  x_n-\floor{x_n} \in I   \}.
$$

\begin{definition}
The sequence $X=\{x_n\}_{n=1}^{\infty}$ is {\it uniformly distributed mod $1$} if for every pair $a,b$ of real numbers with $0 \leq a<b < 1$, we have
\begin{equation}
\lim_{N \to \infty} \frac {A_N([a,b],X)} {N}=\la([a,b])=b-a.
\end{equation}\end{definition}

\begin{theorem}\labt{normalisud}
A real number $x \in [0,1)$ is normal in base $b$ if and only if the sequence $\{b^nx\}_{n=0}^{\infty}$ is uniformly distributed mod $1$.
\end{theorem}



\begin{theorem}\labt{nlank}
Let $\al' < \al < 1$.  Then every $\al$-winning set is $\al'$-winning.
\end{theorem}

\begin{theorem}\labt{blank}
The only $\al$-winning set $S \subset \mathbb{R}$ with $\al>1/2$ is $S=\mathbb{R}$.
\end{theorem}

W. Schmidt proved the following in \cite{SchmidtGames}:

\begin{theorem}\labt{Schmidt1}
Let $0<\al<1$, $0<\be<1$, $\ga=1+\al\be-2\al>0$.  Let $b$ be an integer so large that $b>4/(\al\be\ga)$ and let $d$ be an integer in $[0,b-1]$.  Then the set of all real numbers whose $b$-ary expansion has finitely many occurences of the digit $d$ is $\ab$-winning.
\end{theorem}

\begin{corollary}\labc{SchmidtNormal}
Let $b \geq 2$ be an integer and let $S$ be the set of numbers not normal in base $b$.  Then $S$ is a $1/2$-winning set.
\end{corollary}

Based on \reft{nlank} and \reft{blank}, \refc{SchmidtNormal} is as strong as we could hope for. 

\subsection{Conventions and definitions}

The following conventions and definitions will hold for the rest of this paper.  Put
$$
S=\{ (x,y) \in (0,1)^2: 1+xy-2x>0\}.
$$
Given any $(\al,\be) \in S$, we set $\ga=\ga(\al,\be)=1+\al\be-2\al$.
If $I=(a,b)$ and $J=(c,d)$ with $c>b$, then we say that the {\it distance} between $I$ and $J$ is $c-b$.  We will use a similar definition for closed and half-open intervals.  Additionally, we will say that $I$ {\it intersects} $J$ if $I \cap J \neq \emptyset$.

Let $\mc=\{\mc_k\}_{k=1}^{\infty}$ where 
$\mc_k=\{C_{k,n}\}_{n=-\infty}^{\infty}$ is a sequence of finite disjoint intervals all separated by some positive distance where $C_{k,n-1}$ is positioned to the left of $C_{k,n}$.  Additionally, we require that every finite interval contained in $\mathbb{R}$ intersects at most finitely many members of $\mc_k$ and that for every interval $I$, there exists an integer $K_I$ such that for all $k>K_I$, $I$ intersects at least two members of $\mc_k$.
Let $\md=\{\md_k\}_{k=1}^{\infty}$ where  $\md_k=\{D_{k,n}\}_{n=-\infty}^{\infty}$ is the sequence of intervals between each adjacent member of $\mc_k$ such that $D_{k,n}$ is between $C_{k,n-1}$ and $C_{k,n}$.
Set
$$
\mw_{k,n}=\{ C_{k+1,m} : C_{k+1,m} \subset D_{k,n}\} \hbox{ and } \mv_{k,n}=\{ D_{k+1,m} : D_{k+1,m} \subset D_{k,n}\}
$$
and assume that for all $k$ and $n$, $C_{k+1,n}$ is contained in some member of $\mc_k$ or some member of $\md_k$ and $D_{k+1,n}$ is also contained in some member of $\mc_k$ or some member of $\md_k$
Then we say that $\mc$ is {\it $\ab$-friendly} if each member of $\md_k$ contains at least three members of $\mc_{k+1}$ and there
exists an integer $K_\mc$ such that for all $k > K_\mc$, integers $n$, and intervals $D \in \mv_{k,n}$, we have
\begin{equation}\labeq{friendly1}
\la(D_{k,n}) > \frac {1} {\al\be\ga} \cdot \max(\la(C_{k,n}),\la(C_{k,n+1}))
\end{equation}
and
\begin{equation}\labeq{friendly2}
\la(D) < \frac {(\al\be)^2\ga} {6} \min(\la(D_{k,n-1}),\la(D_{k,n}),\la(D_{k,n+1})).
\end{equation}

Put
$$\xfi=\left\{x \in \mathbb{R} : \left\{k \in \mathbb{N} : x \in \bigcup_{n=-\infty}^{\infty} C_{k,n} \right\} \hbox{ is finite} \right\}.$$

\begin{theorem}\labt{mainwinning}
If $\mc$ is $\ab$-friendly, then $\xfi$ is an $\ab$-winning set.
\end{theorem}

Although it isn't a direct generalization of \reft{Schmidt1}, \reft{mainwinning} may be used to study a much wider variety of sets than \reft{Schmidt1}.
We use \reft{mainwinning} to prove the following:

\begin{enumerate}
\item The set of numbers not normal in some base $b$ is $1/2$-winning.  This result also follows directly from \reft{Schmidt1}.

\item (\reft{cantorwinning1} and \reft{cantorwinning3}) If $Q$ is infinite in limit, then the set of real numbers that are not $Q$-ratio normal of order $2$ and the set of numbers that are not $Q$-distribution normal are both $1/2$-winning.  If $Q$ is $1$-divergent, then the set of numbers that are not simply $Q$-normal is $1/2$-winning.  For every basic sequence $Q$ , the set of numbers that is not strongly $Q$-distribution normal is $1/2$-winning.
\end{enumerate}

\section{Proof of \reft{mainwinning}}

We will need the following lemma from \cite{SchmidtGames}:

\begin{lemma}\labl{schmidt}
Suppose that $(\al,\be) \in S$ and let the integer $t$ satisfy $(\al\be)^t < \ga/2$.  Assume that a ball $B_k$ with radius $\rho_k$ and center $b_k$ occurs in some $(\al,\be)$-play.  Then Alice can play in such a way that $B_{k+t} \subset (b_k+\rho_k \ga/2,\infty)$.
\end{lemma}

\begin{lemma}\labl{maxabg}
If $\ab \in S$, then $\al\be\ga<1/4$.
\end{lemma}

\begin{proof}
Let $f(x,y)=xy(1+xy-2x)=xy+x^2y^2-2x^2y$ for $(x,y) \in [0,1]^2$.  By a routine computation, the maximum of $f(x,y)$ is $1/4$ and occurs only at $(x,y)=(1/2,1)$.  Since $S \subset [0,1]^2$ and $(1/2,1) \notin S$, $\al\be\ga < 1/4$ for $\ab \in S$.
\end{proof}

\begin{lemma}\labl{mainlemma}
Suppose that $\mc$ is $\ab$-friendly and that Bob has chosen $B_n$. If the positive integer $k>K_\mc$ is such that $B_n$ intersects at least two distinct members of $\mc_k$, then there exists an integer $s>n$ so that Alice can play such that $B_s$ intersects no members of $\mc_k$ and $B_s$ intersects at least two members of $\mc_{k+1}$.
\end{lemma}

\begin{proof}
Alice and Bob can play however they want until step $m > n$ where $B_{m-1}$ has a non-empty intersection with adjacent intervals $C_{k,e-1}$ and $C_{k,e}$ in $\mc_k$, but $B_m$ intersects at most one of $C_{k,e-1}$ and $C_{k,e}$ and no other members of $\mc_k$.  
Let $g=\max(\la(C_{k,e-1}),\la(C_{k,e}))$.
 Thus, $\la(B_{m-1})\geq \la(D_{k,e})$ and 
$$
\la(B_m)\geq \al\be \la(D_{k,e}).
$$
  By \refeq{friendly1}, $\la(D_{k,e})>\frac {1} {\al\be\ga} g$ and $\la(B_m)>\al\be \left(\frac {1} {\al\be\ga} g \right)=\frac {1} {\ga} g$, so
\begin{equation}\labeq{qwerty}
g< \ga \la(B_m).
\end{equation}
Thus, Alice needs to worry about avoiding at most one interval, $C$, of length no more than $g$.  Without loss of generality,
\footnote{This is safe to do as \refeq{friendly2} will allow us control over the members of $\mv_{k,e-1}, \mv_{k,e}$, and $\mv_{k,e+1}$ and we will still be able to guarantee that $B_s$ intersects two members of $\mc_{k+1}$.}
 assume that the center of $C$ is less than or equal to the center $b_m$ of $B_m$ and that $C \subset C_{k,e-1}$.  Then by \refeq{qwerty},
$$
C \subset (-\infty,b_m+g/2] \subset (-\infty,b_m+\la(B_m)\ga/2).
$$
Let $t$ be the positive integer that satisfies $\frac {\al\be\ga} {2} \leq (\al\be)^t < \frac {\ga} {2}$ and set $s=m+t$.
By \refl{schmidt}, Alice can play in such a way that
$$
B_s \subset (b_m+\la(B_m)\ga/2,\infty),
$$
so $B_s$ does not intersect any member of $\mc_k$.  We further note that
$$
\la(B_s)=(\al\be)^t \la(B_m) \geq \frac {\al\be\ga} {2} \la(B_m)\geq \frac {(\al\be)^2\ga} {2} \la(D_{k,e}).
$$
Thus, if $D \in \mv_{k,e}$, then by \refeq{friendly2}
$$
\la(B_s) > \frac {(\al\be)^2\ga} {2} \cdot \left(\frac {6} {(\al\be)^2\ga} \la(D)  \right)=3\la(D).
$$
By \refeq{friendly1} and \refl{maxabg}, the length of each member of $\mw_{k,e}$ is less than one fourth of the length of each adjacent member of $\mv_{k,e}$.  Since the length of $B_s$ is at least $3$ times bigger than every gap between members of $\mw_{k,e}$, $B_s$ must intersect at least two members of $\mw_{k,e}$.
\end{proof}

We may now prove \reft{mainwinning}:

\begin{proof}
First, Bob will pick a closed interval $B_1$.  Since $\mc$ is $\ab$-friendly, there is a $k>K_\mc$ such that $B_1$ will intersect at least two members of $\mc_k$.  By \refl{mainlemma}, there exists a positive integer $s$ such that Alice may force some $B_s \cap C_{k,t}=\emptyset$ for all $t$.  But \refl{mainlemma} guarantees that $B_s$ will satisfy the hypotheses of \refl{mainlemma}. So Alice may force $B_r \cap C_{r,t}=\emptyset$ for all $r \geq k$ and integers $t$.  Thus, $\xfi$ is $\ab$-winning.
\end{proof}

\section{Applications}

We will need the following lemma:

\begin{lemma}\labl{halfwinning}
If $S$ is an $(\al,\be)$-winning set for all $(\al,\be) \in D$, then $S$ is $1/2$-winning.
\end{lemma}

\begin{proof}
Suppose that $(1/2,\be) \in (0,1)^2$.  Then $\gamma=\be/2>0$, so $(1/2,\be) \in D$.  Thus, we may conclude that $S$ is a $1/2$-winning set.
\end{proof}

\subsection{The $b$-ary expansions}

As a warmup, we prove that the set of non-normal numbers with respect to the $b$-ary expansion is $1/2$-winning as a consequence of \reft{mainwinning}.  This result was originally found in \cite{SchmidtGames}.

\begin{theorem}\footnote{The proof of this theorem is similar to the proof of \refc{SchmidtNormal}, found in \cite{SchmidtGames}.}
Suppose that $b \geq 2$ is an integer.  Then the set of numbers that are not normal in base $b$ is $1/2$-winning
\end{theorem}

\begin{proof}
Let $\ab \in S$ and let $m$ be large enough so that $b^m>\frac {6} {(\al\be)^2\ga}$. Put $\eta=b^m$.
Define $T_\eta:\mathbb{R} \to [0,1)$ by $T_\eta x= \eta x \pmod{1}$.  Note that for $I=[c,d) \subset [0,1)$,
\begin{equation}\labeq{Ckbary}
T_\eta^{-k} (I)=\bigcup_{n=-\infty}^{\infty} \left[ \frac {n+c} {\eta^k},\frac {n+d} {\eta^k} \right).
\end{equation}
Set $I=[0,1/\eta)$, so $\la(I)=1/\eta < \frac {(\al\be)^2\ga} {6}< \frac {\al\be\ga} {1+\al\be\ga}$.
Let $\mc_k$ consist of the intervals in \refeq{Ckbary} and put $\mc=\{\mc_k\}$.  The intervals in \refeq{Ckbary} all have length of $\eta^{-k} \la(I)$ and are separated by a distance of $\eta^{-k} (1-\la(I))$.  Thus, for all $k$ and $n$, $\la(C_{k,n})=\eta^{-k} \la(I)$ and $\la(D_{k,n})=\eta^{-k} (1-\la(I))$, so \refeq{friendly1} and \refeq{friendly2} both hold.  So, $\xfi$ is $\ab$-winning.
Since $\ab \in S$ was arbitrary and $\xfi$ is always contained in the set of numbers not normal in base $b$, this set is $1/2$-winning by \refl{halfwinning}.
\end{proof}

\subsection{The Cantor series expansion}
\label{sec:4}

The $Q$-Cantor series expansion, first studied by Georg Cantor in \cite{Cantor}, is a natural generalization of the $b$-ary expansion.

\begin{definition}\labd{1.4} $Q=\{q_n\}_{n=1}^{\infty}$ is a {\it basic sequence} if each $q_n$ is an integer greater than or equal to $2$.
\end{definition}

\begin{definition}\labd{1.5} Given a basic sequence $Q$, the {\it $Q$-Cantor series expansion} of a real number $x$ is the (unique)\footnote{Uniqueness can be proven in the same way as for the $b$-ary expansion.} expansion of the form
\begin{equation} \label{eqn:cseries}
x=\floor{x}+\sum_{n=1}^{\infty} \frac {E_n} {q_1 q_2 \ldots q_n}
\end{equation}
such that $E_n$ is in $\{0,1,\ldots,q_n-1\}$ for all $n$
with $E_n \neq q_n-1$ infinitely often.
\end{definition}

Clearly, the $b$-ary expansion is a special case of \refeq{cseries} where $q_n=b$ for all $n$.  If one thinks of a $b$-ary expansion as representing an outcome of repeatedly rolling a fair $b$-sided die, then a $Q$-Cantor series expansion may be thought of as representing an outcome of rolling a fair $q_1$ sided die, followed by a fair $q_2$ sided die and so on.  For example, if $q_n=n+1$ for all $n$, then the $Q$-Cantor series expansion of $e-2$ is
$$
e=2+\frac{1} {2}+\frac{1} {2 \cdot 3}+\frac{1} {2 \cdot 3 \cdot 4}+\ldots
$$
If $q_n=10$ for all $n$, then the $Q$-Cantor series expansion for $1/4$ is
$$
\frac {1} {4}=\frac{2} {10}+\frac {5} {10^2}+\frac {0} {10^3}+\frac {0} {10^4}+\ldots
$$

For a given basic sequence $Q$, let $N_n^Q(B,x)$ denote the number of times a block $B$ occurs starting at a position no greater than $n$ in the $Q$-Cantor series expansion of $x$. Additionally, define
$$
Q_n^{(k)}=\sum_{j=1}^n \frac {1} {q_j q_{j+1} \ldots q_{j+k-1}}.
$$

\begin{definition}\labd{1.7} A real number $x$ is {\it $Q$-normal of order $k$} if for all blocks $B$ of length $k$,
$$
\lim_{n \rightarrow \infty} \frac {N_n^Q (B,x)} {Q_n^{(k)}}=1.
$$
We say that $x$ is {\it $Q$-normal} if it is $Q$-normal of order $k$ for all $k$.
A real number $x$ is {\it $Q$-ratio normal} of order $k$ if for all blocks $B$ and $B'$ of length $k$, we have
$$
\lim_{n \to \infty} \frac {N_n^Q(B,x)} {N_n^Q(B',x)}=1.
$$
$x$ is {\it $Q$-ratio normal} if it is $Q$-ratio normal of order $k$ for all positive integers $k$.  $x$ is {\it simply $Q$-normal} if it is $Q$-normal of order $1$ and {\it simply $Q$-ratio normal} if it is $Q$-ratio normal of order $1$.
\end{definition}

Let~$x$ be a real number and let~$Q$ be a basic sequence.
Define the $1$-periodic function $T_{Q,n}:\mathbb{R} \to [0,1)$ by 
$T_{Q,n}(x)=q_1\cdots q_n x \pmod{1}$.  Given a basic sequence $Q$ and an interval $I \subset [0,1)$, we let $\mc_{Q,I,k}$ consist of the intervals of $T_{Q,k}^{-1}(I)$ and put $\mc_{Q,I}=\{\mc_{Q,I,k}\}_{k=1}^{\infty}$.

\begin{definition}\labd{distributionnormal}
A real number~$x$ is {\it $Q$-distribution normal} if
the sequence $\{T_{Q,n}(x)\}_{n=1}^\infty$ is uniformly distributed mod $1$. $x$ is strongly $Q$-distribution normal if for all positive integers $k$ and $p \in [1,k]$, the sequence $\left\{ T_{Q,kn+p}(x)\right\}_{n=1}^\infty$ is uniformly distributed mod $1$.\footnote{All strongly $Q$-distribution normal numbers are also $Q$-distribution normal.  This follows as the superposition of a finite number of sequences that are uniformly distributed mod $1$ is also uniformly distributed mod $1$ (see \cite{KuN}).  For some basic sequences $Q$, there exist numbers that are $Q$-distribution normal, but not strongly $Q$-distribution normal. For example, put $E=(0,1,0,2,1,3,0,3,1,4,2,5,\ldots)$, $Q=(2,2,4,4,4,4,6,6,6,6,6,6,\ldots)$, and $x=\formalsum$.  Then $x$ is $Q$-distribution normal, but not strongly $Q$-distribution normal as $T_{Q,2n}(x)>1/2$ for all $n$.}
\end{definition}

\begin{definition}\labd{1.8} A basic sequence $Q$ is {\it $k$-divergent} if
$\lim_{n \rightarrow \infty} Q_n^{(k)}=\infty$.
$Q$ is {\it fully divergent} if $Q$ is $k$-divergent for all $k$.
$Q$ is {\it infinite in limit} if $q_n \rightarrow \infty$.
\end{definition}

For $Q$ that are infinite in limit,
it has been shown that the set of all real numbers $x$ that are $Q$-normal of order $k$ has full Lebesgue measure if and only if $Q$ is $k$-divergent \cite{Renyi}.  Therefore, if $Q$ is infinite in limit, then the set of all $x$ that are $Q$-normal has full Lebesgue measure if and only if $Q$ is fully divergent.   We need the following from \cite{KuN}:

\begin{theorem}\labt{weyl}
Let $\{a_n\}_{n=1}^{\infty}$ be a given sequence of distinct integers.  Then the sequence $\{a_nx\}_{n=1}^\infty$ is uniformly distributed mod $1$ for almost all real numbers $x$.
\end{theorem}

The typicality of strongly $Q$-distribution normal numbers follows from \reft{weyl}. Clearly, all numbers that are $Q$-normal of order $k$ are also $Q$-ratio normal of order $k$.  However, unlike the $b$-ary expansion, neither $Q$-normality or $Q$-distribution normality imply each other (see \cite{AlMa} for explicit constructions).

\begin{lemma}\labl{basicseqfriendly}
Suppose that $\ab \in S$ and $Q$ is a basic sequence such that $q_{k}>\frac {6} {(\al\be)^2 \ga}$ for large enough $k$ and that  $I=[a,b) \subset [0,1)$ with $\la(I)<\frac {\al\be\ga} {1+\al\be\ga}$.  Then $\mc_{Q,I}$ is $\ab$-friendly.
\end{lemma}

\begin{proof}
Note that
$$
T_k^{-1}(I)=\bigcup_{n=-\infty}^{\infty} \left[ \frac {n+a} {\qk}, \frac {n+b} {\qk} \right)
$$
is the union of intervals of length $\la(I) \cdot \frac {1} {\qk}$ separated by a distance of $(1-\la(I)) \cdot \frac {1} {\qk}$.  Let $K$ be large enough such that for $k \geq K$, we have $q_{k+1} > \frac {6} {(\al\be)^2\ga}$.  Then
$$
\frac {1} {q_1q_2 \cdots q_{k+1}} (1-\la(I))=\frac {1} {\qk} \cdot \frac {1} {q_{k+1}} \cdot (1-\la(I)) < \frac {(\al\be)^2\ga} {6} \cdot \frac {1} {\qk} \cdot (1-\la(I)),
$$
so \refeq{friendly2} holds. Next, we see that
$$
(1-\la(I)) \cdot \frac {1} {\qk} > \left(1-\frac {\al\be\ga} {1+\al\be\ga} \right) \frac {1} {\qk}=\frac {1} {1+\al\be\ga} \cdot \frac {1} {\qk}
$$
$$
=\frac {1} {\al\be\ga} \cdot  \frac {\al\be\ga} {1+\al\be\ga}  \cdot \frac {1} {\qk}>\frac {1} {\al\be\ga} \cdot \la(I) \cdot \frac {1} {\qk},
$$
so \refeq{friendly1} holds and $\mc_{Q,I}$ is $\ab$-friendly.
\end{proof}

\begin{lemma}\labl{Qinfblocks}
If $Q$ is infinite in limit, $x$ is $Q$-ratio normal of order $2$, and $t$ is a non-negative integer, then
$$\lim_{n \to \infty} N_{n}^Q((t),x)=\infty.$$
\end{lemma}

\begin{proof}
Since $Q$ is infinite in limit and $x$ is $Q$-ratio normal of order $2$, for all $i,j \geq 0$, we have
$$
\lim_{n \to \infty} \frac {N_{n}^Q((t,i),x)} {N_{n}^Q((t,j),x)}=1.
$$
So, for all $j$ there is an $n$ such that $N_{n}^Q((t,j),x) \geq 1$.  Since there are infinitely many choices for $j$, the lemma follows.
\end{proof}

Given $\ab \in S$, put $I_{\ab}=\left[0,\frac {\al\be\ga} {1+\al\be\ga} \right)$.
Let $P_Q$ be the set of real numbers whose $Q$-Cantor series expansion contains finitely many copies of the digit $0$.

\begin{lemma}\labl{pq}
If $Q$ is infinite in limit, then $P_Q$ is $1/2$-winning.
\end{lemma}

\begin{proof}
Let $\ab \in S$ and $K$ large enough so that $q_k > \frac {6} {(\al\be)^2\ga}$ for $k>K$.  Note that $\frac {1} {q_k} < \la(I_{\ab})$, so if the $j^{th}$ digit of some real number $x$ is $0$, then $T_{Q,j-1}(x) \in I_{\ab}$.  But $\mc_{Q,I_{\ab}}$ is $\ab$-friendly and $\ab \in S$ was arbitrary, so $P_Q$ is $1/2$ winning by \reft{mainwinning}, \refl{halfwinning}, and \refl{basicseqfriendly}.
\end{proof}

\begin{lemma}\labl{pqnn}
If $Q$ is infinite in limit, then $P_Q$ is contained in the set of real numbers that are not $Q$-ratio normal of order $2$.
\end{lemma}

\begin{proof}
If $x \in P_Q$, then the digit $0$ occurs finitely often.  If there is another digit that occurs a different number of times, then $x$ is not even simply $Q$-ratio normal.  If every digit in the $Q$-Cantor series expansion of $x$ occurs the same number of times, then $x$ is not $Q$-ratio normal of order $2$ by \refl{Qinfblocks}.
\end{proof}

\begin{theorem}\labt{cantorwinning1}
Suppose that $Q$ is infinite in limit. Then the set of numbers that are not $Q$-ratio normal of order $2$ is a $1/2$-winning set.  
The set of real numbers that are not $Q$-distribution normal is $1/2$-winning.
If $Q$ is $1$-divergent, then the set of numbers that are not simply $Q$-normal is $1/2$-winning.
\end{theorem}

\begin{proof}
The first conclusion follows directly from \refl{pq} and \refl{pqnn}. If $Q$ is $1$-divergent and $x$ is simply $Q$-normal, then every digit occurs infinitely often in the $Q$-Cantor series expansion of $x$. Thus, the set of numbers that is not simply $Q$-normal is $1/2$ winning by \refl{pq}. Additionally, the set of real numbers that are not $Q$-distribution normal contains $\mathscr{X}_{\mathscr{C}_{Q,I_{\ab}}}$ for all $\ab \in S$ and is $1/2$-winning as well.




\end{proof}

\begin{corollary}\labc{cantorwinning2}
If $Q$ is infinite in limit, then the set of numbers that are not $Q$-normal of order $2$ is a $1/2$-winning set.
\end{corollary}

\begin{theorem}\labt{cantorwinning3}
If $Q$ is any basic sequence, then the set of numbers that are not strongly $Q$-distribution normal is a $1/2$-winning set.
\end{theorem}

\begin{proof}
Suppose that $\ab \in S$.  Put $t=1+\ceil{\log_2 \frac {6} {(\al\be)^2\ga}}$.
Let the basic sequence $\Psi_{Q,t}=\{\psi_{t,j}\}_{j=1}^{\infty}$ be given by
$\psi_{t,j}=q_{(j-1)t+1} \cdot q_{(j-1)t+2} \cdots q_{jt}$.
Thus, since $q_n \geq 2$ for all $n$, $\psi_{t,j} \geq 2^t>6/((\al\be)^2\ga)$.  Then $\mathscr{X}_{\mathscr{C}_{\Psi_{Q,t},I_{\ab}}}$ is $\ab$-winning.
But $\mathscr{X}_{\mathscr{C}_{\Psi_{Q,t},I_{\ab}}}$  is contained in the set of numbers that are not strongly $Q$-distribution normal, so the conclusion follows.
\end{proof}

\begin{corollary}
The Hausdorff dimension of all of the sets considered in \reft{cantorwinning1}, \refc{cantorwinning2}, and \reft{cantorwinning3} is $1$.
\end{corollary}

\begin{remark}
The proof of \reft{cantorwinning1} also shows that if $Q$ is infinite in limit, then the set of real numbers $x$ such that $\{\qn x \pmod{1} \}_{n=1}^{\infty}$ is not dense in $[0,1)$ is a $1/2$-winning set.  Similar results can be stated for the sets considered in \reft{cantorwinning1} and \reft{cantorwinning3}.
\end{remark}

\bibliographystyle{amsplain}

\begin{thebibliography}{10}




\bibitem{AlMa}  C. Altomare, B. Mance, {\em Cantor series constructions contrasting two notions of normality}, Monatsh. Math. (to appear).

\bibitem{Borel} E. Borel, {\em Les probabilit\'es d\'enombrables et leurs applications arithm\'etiques}, Rend. Circ. Mat. Palermo {\bf  27} (1909), pp. 247--271.

\bibitem{Cantor} G. Cantor, {\em \"Uber die einfachen Zahlensysteme}, Zeitschrift f\"ur Math. und Physik {\bf  14} (1869), pp. 121--128.

\bibitem{Champernowne}  D. G. Champernowne, {\em The construction of decimals normal in the scale of ten}, Journal of the London Mathematical Society {\bf  8} (1933), pp. 254--260.

\bibitem{Dani} S. G. Dani.  {\em On badly approximable numbers, Schmidt games and bounded orbits of flows}.  In {\it Number theory and dynamical systems}, London Math. Soc. Lecture Note Ser. {\bf 134}, pp. 69--86, York, 1987.

\bibitem{DT} M. Drmota, R. F. Tichy, {\em Sequences, Discrepancies and Applications}, Springer-Verlag, Berlin Heidelberg (1997).




\bibitem{KuN} L. Kuipers, H. Niederreiter,  {\em Uniform Distribution of Sequences}, Dover, Mineola, NY, 2006.


\bibitem{Renyi} A. R\'enyi,  {\em On the distribution of the digits in Cantor's series}, Mat. Lapok {\bf 7} (1956), pp. 77--100.


\bibitem{SchmidtGames} W. M. Schmidt, {\em On badly approximable numbers and certain games}, Trans. A.M.S. {\bf 123} (1966), pp. 27--50.


\end{thebibliography}

\end{document}